\documentclass[12pt,a4paper,final,reqno]{amsart}
\usepackage{amsmath}
\usepackage{amsthm}
\usepackage{amssymb}
\usepackage{graphicx}
\usepackage{mathrsfs}
\usepackage[english, portuguese]{babel}
\usepackage[latin1]{inputenc}
\usepackage[T1]{fontenc}
\usepackage{amsrefs}
\usepackage[usenames,dvipsnames]{xcolor}

\def\real{\mathbb{R}}

\renewcommand{\d}{\mathrm d}

\def\real{{\mathbb R}}

\def\i{{\mathfrak{i}}}

\renewcommand{\d}{d}

\newcommand{\Poi}{{\rm Pois}}
\newcommand{\Pn}{{\rm Pn}}
\newcommand{\Sp}{{\rm Sp}}

\newcommand{\SL}{{\rm SL}}
\newcommand{\GL}{{\rm GL}}
\newcommand{\SO}{{\rm SO}}
\newcommand{\Rot}{{\rm Rot}}

\newcommand{\id}{{\rm id}}

\newcommand{\Jj}{\mathbb J}
\newcommand{\Nn}{\mathbb N}
\newcommand{\Rr}{\mathbb R}

\newcommand{\EE}{\mathcal E}

\newcommand{\comment}[1]{}

\numberwithin{equation}{section}

\newtheorem{theorem}{Theorem}[section]
\newtheorem{lemma}[theorem]{Lemma}

\newtheorem{proposition}[theorem]{Proposition}

\theoremstyle{remark}
\newtheorem{remark}[theorem]{Remark}

\theoremstyle{definition}

\begin{document}
\selectlanguage{english}

\title[]{Realization of tangent perturbations in discrete and continuous time conservative systems}

\author[Alishah]{Hassan Najafi Alishah}
\address{CMAF \\
Faculdade de Ci\^encias\\
Universidade de Lisboa\\
Campo Grande, Edif\'icio C6, Piso 2\\
1749-016 Lisboa, Portugal}
\email{najafi.alishah@gmail.com}

\author[Lopes Dias]{Jo\~ao Lopes Dias}
\address{Departamento de Matem\'atica and CEMAPRE, ISEG\\
Universidade de Lisboa\\
Rua do Quelhas 6, 1200-781 Lisboa, Portugal}
\email{jldias@iseg.utl.pt}

\begin{abstract}
We prove that any perturbation of the symplectic part of the derivative of a Poisson diffeomorphism can be realized as the derivative of a $C^1$-close Poisson diffeomorphism.
We also show that a similar property holds for the Poincar\'e map of a Hamiltonian on a Poisson manifold.
These results are the conservative counterparts of the Franks lemma, a perturbation tool used in several contexts most notably in the theory of smooth dynamical systems.

\end{abstract}

\maketitle


\section{Introduction}\label{sec:introduction}

The so-called Franks lemma~\cite[Lemma 1.1]{Franks} states that any perturbation of the derivative of a diffeomorphism at a finite set can be realized as the derivative of a nearby diffeomorphism in the $C^1$ topology.
This perturbative result has been crucially used to produce dynamical results out of related properties for linear systems.
Many different dynamical behaviors can then be proved to hold in dense or even residual sets of diffeomorphisms. 
In the case of flows similar perturbation techniques are contained in~\cite{bgv2006,mpp2004}.

Based on the usefulness of the Franks lemma, it is natural to ask if it still holds by restricting its focus to certain subgroups of diffeomorphisms. 
In the volume-preserving context Bonatti, D\'\i az and Pujals~\cite[Proposition 7.4]{bdp2003} proved a version of the Franks lemma for diffeomorphisms and Bessa and Rocha~\cite[Lemma 3.2]{BR2008} for flows.
In the symplectic case some authors have stated and used it (see e.g.~\cite[Lemma 12]{arnaud} and~\cite[Lemma 5.1]{HA2006}), but up to our knowledge no proof is available in the literature.

In this paper we present a complete proof of the symplectic perturbative result version, as a particular case of a slightly more general setting concerning Poisson diffeomorphisms.
More specifically, we show that a perturbation of the symplectic part of the derivative of a Poisson diffeomorphism at a point $p$ can be realized as the derivative of a nearby Poisson diffeomorphism which differs from the original map only at a small neighborhood of $p$.

We also show the similar result for general Hamiltonians in Poisson manifolds as a simple application of the ideas used for symplectomorphisms (this simplifies considerably the methods described in the manuscript~\cite{vivier}).
In fact, we show that a linear perturbation of the derivative of the Poincar\'e map is realizable as the derivative of the Poincar\'e map of a nearby Hamiltonian.
When considering geodesic flows see~\cite{ContrerasPaternain2002} for surfaces and~\cite{Contreras2010} for a higher dimensional manifold (see also~\cite{visscher2012}).

The fact that every Poisson diffeomorphism has to preserve the symplectic foliation of a Poisson manifold is an obstruction to state a general Franks lemma for Poisson maps. 
However, the type of perturbation we study here includes time one maps of Hamiltonian flows on Poisson manifolds. 
The Hamiltonian dynamical systems on Poisson manifolds, sometimes referred as generalized Hamiltonian systems, arise naturally in problems of celestial mechanics, mean field theory, ecology populations, among many others (cf. e.g. \cite{MR1643501,MR1672182,MR660413} and references therein).

\subsection{Organization of the paper} 

In section~\ref{sec:results} we provide basic definitions on Poisson manifolds and state our main results, Theorems~\ref{thm:Frankspo} and ~\ref{thm:Franks_symp}, on Poisson diffeomorphisms and the Poincar\'e map of Hamiltonians, respectively. 
Section~\ref{sec:prelims} contains the key technical lemma regarding the Hamiltonian function that we will use to achieve perturbations for the special case of rotations. 
The linear symplectic geometry techniques to reduce a general case to rotations will be provided in the last part of section~\ref{sec:prelims}.
We will prove Theorem~\ref{thm:Frankspo} in section~\ref{sec:proof1}. 
At the beginning of section~\ref{sec:Hamiltonian} we will show Theorem~\ref{thm:Franks_symp}, first in a simpler case and later, through a Poisson flowbox theorem (Theorem~\ref{thm:flowbox}), for the general setting.

\section{Statement of results}\label{sec:results}

A Poisson manifold is a pair $(M,\pi)$ where $M$ is a manifold and $\pi$ a Poisson structure on $M$.
Recall that a Poisson structure is a bivector field $\pi\in\mathfrak{X}_2(M)$ with the property that $[\pi,\pi]=0$, where $[\cdot, \cdot]$ is the Schouten bracket  (cf. e.g.~\cite{poissonbook,MR1240056}).
The bivector field $\pi$ provides a vector bundle map $\sharp_\pi\colon T^\ast  M\rightarrow TM$. 
The image of this map is an integrable singular distribution which integrates to a symplectic foliation, i.e. a foliation whose leaves have a symplectic structure induced by the Poisson structure. 
The rank of Poisson structure at $p\in M$ is half of the dimension of the symplectic leaf passing through $p$.

Notice that a Poisson structure can be also defined as a Lie bracket $\{\cdot,\cdot\}$ on $C^\infty(M)\times C^\infty(M)$ satisfying the Leibniz identity
$$
\{\psi,\phi\, \eta\}=\{\psi,\phi\}\eta+\phi\{\psi,\eta\},
\qquad 
\psi,\phi,\eta\in C^\infty(M).
$$
The two above descriptions are related by $\pi(\d \psi,\d \phi)=\{\psi,\phi\}$.

The set $\Poi^1(M)$ of Poisson diffeomorphisms consists of $C^1$-diffeomorphisms $f\colon M\to M$ satisfying $f_\ast\pi=\pi$, where 
$$
f_\ast\pi(\xi,\eta)=\pi(f^\ast\xi,f^\ast\eta),
\qquad 
\xi,\eta\in T^\ast M.
$$
For the Poisson bracket description, $f$ is Poisson iff $\{\psi\circ f,\phi\circ f\}=\{\psi,\phi\}$ for every $\psi,\phi \in C^\infty(M)$.

A regular Poisson manifold is a Poisson manifold with constant rank. 
We now restrict our attention to regular Poisson manifolds $(M,\pi)$ with rank $d$ and dimension $2d+n$.
By the splitting theorem, the Poisson version of the Darboux theorem (see e.g.~\cite{poissonbook}), there is an atlas $\{(U_i,\varphi_i)\}_{i\in\Nn}$, such that 
$(\varphi_i)_\ast\pi=\pi_0$, where
$$
\pi_0=\sum_{i=1}^d\frac{\partial}{\partial x_i}\wedge\frac{\partial}{\partial y_i}
$$
is the canonical Poisson structure.
Here, $(x_1,..,x_d,y_1,...,y_d,z_1,..,z_n)$ stands for the coordinates of $\Rr^{2d}\times\Rr^n$. 
We will always use local coordinates given by the splitting theorem so that the derivative of $f\in \Poi^1(M)$ at $p\in M$, $D_pf=D_0 (\varphi_j\circ  f\circ\varphi_i^{-1})$, belongs to the Poisson linear group given by
$$
\Pn(2d+n,\Rr)=\{B\in\GL(2d+n)\colon B^T\hat\Jj B=\hat\Jj\},
$$
with
$$
\hat\Jj=
\begin{pmatrix}\Jj & 0\\
0&0	\end{pmatrix},
\qquad
\Jj=\begin{pmatrix} 0&-I\\
I&0\end{pmatrix}
$$
and the $d\times d$ identity matrix $I$.
Note that the elements of $\Pn(2d+n,\Rr)$ are of the form
$$
\begin{pmatrix}A&a\\0&b\end{pmatrix}
=\begin{pmatrix}I& a\\0&b\end{pmatrix}\,A_\pi ,
$$
where $A\in \Sp(2d,\real)$ is a symplectic matrix, i.e. $A\in\GL(2d,\Rr)$ such that $A^T\Jj A=\Jj$, $a$ is any $2d\times n$ real matrix, $b\in\GL(n,\Rr)$ and
\begin{equation}\label{def:Api}
A_\pi=\begin{pmatrix}A&0\\0&I\end{pmatrix} \in \Pn(2d+n,\Rr).
\end{equation}

For $f\in \Poi^1(M)$, $\varepsilon>0$ and $D\subset M$, define
$$
B_\varepsilon(f,D)=\left\{g\in\Poi^1(M) \colon \|g-f\|_{C^1}< \varepsilon, \, g=f \text{ on } D \right\}.
$$

\begin{theorem}\label{thm:Frankspo}
Let $\varepsilon>0$, $f\in\Poi^1(M)$ and $p\in M$.
Then, there is $\delta>0$ such that for every neighborhood $V$ of $p$,
$$
\left\{{A}_\pi D_p f\colon {A} \in\Sp(2d,\Rr), \|A-I\| < \delta \right\}
\subset
\left\{D_pg \colon g\in B_\varepsilon(f,D) \right\}
$$ 
where $D=(M\setminus V)\cup\{p\}$.
\end{theorem}

We now focus on the Hamiltonian flow case.
Consider $(M,\pi)$ to be a regular Poisson manifold with rank $d+1$ and dimension $2(d+1)+n$, and $H\in C^2(M)$ a Hamiltonian function.
The map $\sharp_\pi\colon T^\ast  M\to  TM$ associates a Hamiltonian $H\colon M\to \Rr$ to a Hamiltonian vector field by
$$
X_H = \sharp_\pi(\d H),
$$
which generates the Hamiltonian flow $\varphi_H^t$ in $M$.
Let $\EE$ be the energy surface passing through a regular point $p\in M$, i.e. the connected component of $H^{-1}(\{H(p)\})$ containing $p$.
Take $S$ to be the symplectic leaf passing through $p$. 
Around $p$ the set $\EE\cap S$ is a $2d+1$ dimensional submanifold of $M$.  
A \emph{transversal} $\Sigma$ to the flow at $p$ in $\EE\cap S$ is a $2d$-dimensional smooth submanifold verifying
$$
T_p(\EE\cap S) = T_p\Sigma \oplus \Rr X_H(p).
$$
Note that $\Sigma$ is a symplectic submanifold of $S$.

Now, take $p'=\varphi_H^T(p)$ for some $T>0$, a transversal $\Sigma'$ to the flow at $p'$ and some neighborhood $U\subset M$ of $p$.
The \emph{Poincar\'e map} of $H$ at $p$ is defined to be the $C^{1}$-symplectomorphism $P_H\colon \Sigma\cap U\to P_H(\Sigma\cap U)\subset \Sigma'$ given by
$$
P_H(x)=\varphi_H^{\tau(x)}(x)
\quad\text{and}\quad
\tau(x)=\min\{t\geq0\colon \varphi_H^t(x) \in\Sigma'\}.
$$
Notice that $U$ is assumed to be sufficiently small such that, by the implicit function theorem, $\tau$ is $C^{1}$ and $\tau(U)$ is bounded.

The linear Poincar\'e map of $H$ at $p$ is the derivative of the Poincar\'e map at $p$,
$$
D_pP_H\colon T_p \Sigma\to T_{p'}\Sigma',
$$ 
which can be seen as an element of $\Sp(2d,\Rr)$ using local coordinates. 

For $H\in C^2(M)$, $\varepsilon>0$ and $D\subset M$, consider the set
$$
B_\varepsilon(H,D)=\{H'\in C^2(M) \colon \|H'-H\|_{C^2}< \varepsilon, \, X_{H'}=X_H \text{ on } D\}.
$$
Since we want to realize perturbations by Hamiltonians inside $B_\varepsilon(H,\Gamma)$ for an orbit segment $\Gamma$ of $\varphi_H^t$, we fix the transversals $\Sigma$ and $\Sigma'$ taken at $p$ and $p'$ both in $\Gamma$.

\begin{theorem}\label{thm:Franks_symp}
Let $\varepsilon>0$, $H\in C^2(M)$ with an orbit segment $\Gamma$ starting at $p\in M$.
Then, there is $\delta>0$ such that for every tubular neighborhood $W$ of $\Gamma$, 
$$
\left\{A D_pP_H \colon {A} \in\Sp(2d,\Rr)\colon \|{A}-I\| < \delta \right\}
\subset
\left\{D_pP_{H'} \colon H'\in B_\varepsilon(H,D) \right\}
$$
where $D=(M\setminus W)\cup \Gamma$.
\end{theorem}

\section{Preliminaries}\label{sec:prelims}

\subsection{Definitions}

Consider the $\ell^1$-norm on $\Rr^d$ which induces the matrix norm $\|[a_{i,j}]\| = \max_j\sum_i|a_{i,j}|$. The $C^1$-norm given for any $\varphi\in C^1(\Rr^d,\Rr^{d'})$ by
$$
\|\varphi\|_{C^1}=\max \left\{\|\varphi\|_{C^0},  \max_{j} 
\left\| \frac{\partial \varphi}{\partial x_j}\right\|_{C^0}
\right\},
$$
where $\|\varphi\|_ {C^0}=\sum_i\sup_{x\in\Rr^d} |\varphi_i(x)|$ is the uniform norm.


Given $1\leq k\leq d$, consider the embedding
$\pi_k\colon \SL(2,\Rr)\to\Pn(2d+n,\Rr)$ given by $\pi_k(M)=[\alpha_{i,j}]_{i,j=1}^{2d}$ where $\alpha_{i,i}=1$, $i\not\in\{k,d+k\}$, 
$$
\begin{pmatrix}
\alpha_{k,k} &\alpha_{k,d+k}  \\
\alpha_{d+k,k} & \alpha_{d+k,d+k} 
\end{pmatrix}
=M,
$$
and $\alpha_{i,j}=0$ for the remaining cases. Notice that $\pi_k(I)=I$.
We can easily check that $\pi_k$ is a homomorphism $\pi_k(M_1M_2)=\pi_k(M_1)\,\pi_k(M_2)$ and $\|\pi_k(M)\|\leq \max\{1,\|M\|\}$.

Define
$$
\Rot_2(2d,\Rr)=\bigcup_{1\leq k\leq d}\pi_k(\SL(2,\Rr))
$$
as the set of the symplectic matrices that rotate only in two conjugate directions.
Recall the form of the one-parameter group of rotations $\SO(2,\Rr)=\{R_\alpha\colon \alpha\in\Rr\}$, where
$$
R_\alpha=
\begin{pmatrix}
\cos\alpha&-\sin\alpha\\
\sin\alpha&\cos\alpha
\end{pmatrix}.
$$

Consider also the embedding 
\begin{equation}\label{def:Phi}\Phi\colon \Sp(2d,\Rr)\to\Pn(2d+2+n,\Rr),
\end{equation}
 for which $\Phi(M)=[\alpha_{i,j}]_{i,j=1}^{2d+2+n}$
is given by $\alpha_{i,i}=1$, $i\in\{1,d+1,2d+1,\dots,2d+n\}$,
$$
\begin{pmatrix}
\alpha_{2,2} &\dots &\alpha_{2,d+1} & \alpha_{2,d+3} &\dots &\alpha_{2,2d+2} \\
\vdots & & \vdots  & \vdots & & \vdots \\
\alpha_{d+1,2} &\dots &\alpha_{d+1,d+1} & \alpha_{d+1,d+3} &\dots &\alpha_{d+1,2d+2} \\
\alpha_{d+3,2} &\dots &\alpha_{d+3,d+1} & \alpha_{d+3,d+3} &\dots &\alpha_{d+3,2d+2} \\
\vdots & & \vdots  &\vdots  & & \vdots \\
\alpha_{2d+2,2} &\dots &\alpha_{2d+2,d+1} & \alpha_{2d+2,d+3} &\dots& \alpha_{2d+2,2d+2}
\end{pmatrix}
=M,
$$
and $\alpha_{i,j}=0$ for the remaining cases.
We can easily check that it is also a homomorphism $\Phi(M_1M_2)=\Phi(M_1)\,\Phi(M_2)$ and $\|\Phi(M)\|\leq \max\{1,\|M\|\}$.

Finally, choose a bump function $\ell\in C^\infty(\Rr)$ satisfying 
\begin{equation}\label{def:ell}
\ell(r)=
\begin{cases}
1, & |r| \leq 1/4 \\
0, & |r|\geq 1,
\end{cases}
\end{equation}
$0\leq\ell(r)\leq1$, $\int\ell=1$ and $\|\ell\|_{C^2}$ bounded by a universal constant (we fix this value in the following).

\subsection{Hamiltonian  perturbation}

For $r>0$ define $B_r\subset\Rr^{2d+n}$ to be the Euclidean open ball centered at the origin with radius $r$.
We write $\|\cdot\|_2$ for the Euclidean norm.

Given $\alpha\in\Rr$ and $1\leq i\leq d$, consider the $C^\infty$ function $K_i\colon\Rr^{2d+n}\to\Rr$,
\begin{equation}\label{Ham K}
K_i(x,y,z) = 
\alpha\ell\left(\rho\right) \rho_i,
\end{equation}
where
\begin{align*}
\rho & =\frac12 \|( x, y, z )\|_2^2, \\
\rho_i & =\frac12 (x_i^2+y_i^2).
\end{align*}
We will also be using the following notations:
\begin{align*}
\vartheta_i & =\ell'(\rho)\rho_i+\ell(\rho) \\
\vartheta_j & =\ell'(\rho)\rho_i, \quad j\neq i.
\end{align*}

\begin{lemma}\label{lemma K}\hfill
\begin{enumerate}
\item\label{claim 1} $\varphi_{K_i}^t=\pi_i(R_{t\alpha\vartheta_i}) \Pi_{j\neq i}\pi_j(R_{t\alpha\vartheta_j})$,
\item $\|{K_i}\|_{C^2} \leq c|\alpha|$ for some constant $c>0$.
\end{enumerate}
\end{lemma}

\begin{remark}
It is simple to check that $K_i=0$ and $\varphi_{K_i}^t=\id$ on $\Rr^{2d+n} \setminus B_1$.
Moreover, for $r\leq 1/4$ we have 
$$
\varphi_{K_i}^t=\pi_i(R_{t\alpha})
$$
which keeps $B_r$ invariant.
\end{remark}

\begin{proof}
The Hamiltonian vector field $X_{K_i}(x,y,z) = \hat{\Jj} \nabla K_i(x,y,z)$ is given by
$$
\begin{cases}
(\dot x_{i}, \dot y_{i})  = \alpha\vartheta_i (- y_i, x_{i} ),\\
(\dot x_j,\dot y_j)=\alpha\vartheta_j (-y_j,x_j), & j\neq i\\ 
\dot{z}_k=0, & k=1,\dots,n.
\end{cases}
$$
It is easy to check that $\frac d{dt}\rho=\frac d{dt}\rho_i=0$.
This in turn means that $\vartheta_i$ and $\vartheta_j$ are constants of motion as well. 
So, the Hamiltonian flow is as stated in the first claim since it is made of two-dimensional rotations between the symplectic conjugated coordinates.

Now, it is simple to check that
\begin{equation*}
\|{K_i}\|_{C^1} = \max\{ \|K_i\|_{C^0} , \| \nabla K_i\|_{C^0} \} \leq |\alpha| \,\|\ell\|_{C^1}.
\end{equation*}
The second order derivatives are the following:
\begin{equation*}
\begin{split}
\frac{\partial^2K_i}{\partial x_i\partial y_i}& =\alpha  [\ell''(\rho)\rho_i+2\ell'(\rho)] x_iy_i \\
\frac{\partial^2K_i}{\partial w^2} &=\alpha [\ell'(\rho)\rho_i+\ell(\rho) +\ell''(\rho)\rho_i w^2+2\ell'(\rho)w^2] \\
\frac{\partial^2K_i}{\partial w\partial u}& =  \alpha  [\ell''(\rho)\rho_i+\ell'(\rho)]wu \\
\frac{\partial^2K_i}{\partial u^2}& = \alpha [\ell'(\rho) +\ell''(\rho)u^2]\rho_i \\
\frac{\partial^2K_i}{\partial u\partial v}& = \alpha  \ell''(\rho)\rho_i u v, 
\end{split}
\end{equation*}
where $w$ stands for $x_i$ or $y_i$, $u$ and $v$ replace $x_j$, $y_j$ or $z_k$ with $j\not=i$ and $u\not=v$.
So, for some constant $c>0$,
\begin{equation*}
\|D^2K\|_{C^0}  \leq c |\alpha|.
\end{equation*}
\end{proof}

\subsection{Symplectic linear algebra}

Denote the canonical basis of $\Rr^m$ by $\{e_1,...,e_m\}$.

\begin{lemma}\label{lemma:diag}
Let $L=\mathrm{diag}(\lambda_1,\dots,\lambda_m)\in\GL(m,\Rr)$ with distinct eigenvalues and $1\leq i \leq m$. 
Then, any $A\in\GL(m,\Rr)$ close enough to $L$ has a distinct eigenvalue $\tilde{\lambda}_i$ close to $\lambda_i$ with a unique associated eigenvector $v_i$ close to $e_i$ such that $\|v_i\|=1$. 
\end{lemma}

\begin{proof}
Without loss of generality, set $i=1$. 
Define $F\colon  \GL(m,\Rr)\times\Rr\times\Rr^{m}\to\Rr\times\Rr^{m}$ given by
\begin{align*}
F(A,\nu,q)=
\left(
\det(A-\nu I),
(A-\nu I)q+
(q_1^2+\dots+q_m^2-1)e_1
\right)
\end{align*}
Notice that $F(L,\lambda_1,e_1)=0$ and 
$$
\det D_{(\nu,q)}F(L,\lambda_1,e_1)=
\begin{vmatrix}
\Pi_{j\neq1}(\lambda_j-\lambda_1)&0&0&\dots&0\\
-1&2&0&\dots&0\\
0&0& \lambda_2-\lambda_1 &\dots&0\\
\vdots&\vdots&\vdots&\ddots&\vdots \\
0&0&0&\dots& \lambda_m-\lambda_1
\end{vmatrix}
\not=0.
$$
The implicit function theorem now proves the lemma.
\end{proof}

Consider the canonical symplectic form $\omega_0=\sum_{i=1}^d dx_i\wedge dy_{i}$.
A basis 
$$
\{v_1,\dots,v_d,w_1,\dots,w_d\}
$$
of $\Rr^{2d}$ is symplectic if $\omega_0(v_i,w_i)=\delta_{ij}$, where $\delta_{ij}$ is the  Kronecker delta.

\begin{lemma}\label{lemma:diagsymp}
Let $L=\mathrm{diag}(\lambda_1,\dots,\lambda_d,\lambda_1^{-1},\dots ,\lambda_d^{-1})\in\Sp(2d,\Rr)$ with distinct eigenvalues. 
Then, for any $A\in\Sp(2d,\Rr)$ close enough to $L$ there is $S\in\Sp(2d,\Rr)$ close to the identity such that
$$
A=S\,\mathrm{diag}(\tilde{\lambda}_1,\dots,\tilde{\lambda}_d,\tilde{\lambda}_1^{-1},\dots ,\tilde{\lambda}_d^{-1})S^{-1}
$$ 
with distinct eigenvalues close to the respective eigenvalues of $L$.
\end{lemma}

\begin{proof}
Applying Lemma~\ref{lemma:diag} for every $i=1,\dots,d$ one gets a small neighborhood of $L$ in which a matrix $A$ has 
pairs of eigenvalues and normalized eigenvectors 
$$
(\tilde{\lambda}_1,v_1),\dots,(\tilde{\lambda}_d,v_d)
$$ 
close to $(\lambda_1,e_1),\dots,(\lambda_d,e_d)$, respectively. 
The fact that $A$ is symplectic implies that the other pairs of eigenvalues/normalized eigenvectors are
$$
(\tilde{\lambda}^{-1}_1,w_1),\dots,(\tilde{\lambda}^{-1}_{d},w_d),
$$ 
again close to $(\lambda^{-1}_1,e_{d+1}),\dots,(\lambda^{-1}_d,e_{2d})$, respectively. 
The eigenvalues can be made distinct as long as $A$ is sufficiently close to $L$.
Thus, $A$ is diagonalizable by a matrix $S$.
It remains to show that $S$ is symplectic.


Notice that for any pair of eigenvectors $v,v'$ associated to the eigenvalues $\lambda,\lambda'$ we have
$$
\lambda\lambda'\omega_0(v,v')=\omega_0(Av,Av')=\omega_0(v,v').
$$
Hence, if $\lambda\lambda'\neq1$ then $\omega_0(v,v')=\langle v,\Jj v'\rangle=0$.
Since all the eigenvalues are distinct,
$$
\omega_0(v_i,v_j)=\omega_0(w_i,w_j)=0
$$
and $\omega_0(v_i,w_j)=0$ for $i\neq j$.
By the non-degeneracy of $\omega_0$ we conclude that $\omega_0(v_i,w_i)\neq 0$ for all $i=1,\dots,d$. 

Since $v_i$ and $w_i$ are close to $e_i$ and $e_j$, respectively, the scalar $\omega_0(v_i,w_i)$ is close to one. 
Dividing the eigenvectors $v_i$ by $\omega_0(v_i,w_i)$ gives us a symplectic basis of eigenvectors close to the canonical one. This matrix forms the columns of $S^{-1}$ which is therefore close to the identity.
From $S=I-S(S^{-1}-I)$ we get $\|S\|\leq(1-\|S^{-1}-I\|)^{-1}$, which then implies $\|S-I\|=\|S(S^{-1}-I)\|\leq (1-\|S^{-1}-I\|)^{-1}\|S^{-1}-I\|$.
\end{proof}

\begin{lemma}\label{lemma:decomposition symp}
There exist $\epsilon,c>0$ such that any $A\in\Sp(2d,\Rr)$ with $\|A-I\|< \epsilon$ can be written as
$A=A_1 \dots A_{4d}$,
where $A_i=P_i R_i P^{-1}_i$ with $R_i\in\Rot_2(2d,\Rr)$, $\|R_i-I\|<c\|A-I\|^{1/2}$, $P_i\in\Sp(2d,\Rr)$ and  $\|P_i^{\pm}\|\leq c$.
\end{lemma}

\begin{proof}
Our goal is first to write the matrix $A$ as a product of diagonal and diagonalizable matrices. 
Then we will show that diagonal matrices can be written as the product of symplectic rotations in $\Rot_2(2d,\Rr)$ up to symplectic linear conjugacy.

Write $A=L_1 B$, where $B=L_1^{-1} A$ and 
$$
L_1=\mathrm{diag}(\lambda_1,\dots,\lambda_d,\lambda_1^{-1},\dots,\lambda_d^{-1})\in\Sp(2d,\Rr)
$$
with distinct eigenvalues and $|\lambda_i-1|\leq \|A-I\|$. 
 If $\|A-I\|$ is sufficiently small then by Lemma~\ref{lemma:diagsymp} we have $PL_2P^{-1}$ where  $L_2=\mathrm{diag}(\tilde\lambda_1,\dots,\tilde\lambda_d,\tilde\lambda_1^{-1},\dots,\tilde\lambda_d^{-1})\in\Sp(2d,\Rr)$ and $P$ is a symplectic matrix close to the identity. The eigenvalues verify $|\tilde\lambda_i-1|\leq c_1\|A-I\|$ with a constant $c_1>0$.

Any invertible diagonal matrix $L=\mathrm{diag}(\eta_1,\dots,\eta_d,\eta_1^{-1},\dots,\eta_d^{-1})\in\Sp(2d,\Rr)$ can be written as the product of $d$ diagonal symplectic matrices each being essentially two-dimensional:
\begin{equation}\label{matrix:L}
L=\prod_{i=1}^d \pi_i (\mathrm{diag}(\lambda_i,\lambda_i^{-1})).
\end{equation}
This means that we can reduce our setting to two-dimensions to deal with such decompositions for $L_1$ and $L_2$ as given above (corresponding to a total of $2d$ diagonal matrices).

Consider $D=\left(\begin{smallmatrix}\eta_i&0\\0&\eta_i^{-1}\end{smallmatrix}\right)\in\Sp(2,\Rr)=\SL(2,\real)$, where $|\eta_i-1|\leq \max\{1,c_1\} \|A-I\|$, and  $R_\theta\in\SO(2,\real)$ the rotation matrix by an angle $\theta\in[0,\frac\pi{2}]$ with
\begin{equation}\label{equ:theta}
\frac{(\eta_i-1)^2}{\eta_i^2+1}<1-\cos\theta \leq |\eta_i-1|.
\end{equation}

Writing $D=R_\theta R_{-\theta} D$ and recalling that $\pi(D)=\pi_k(R_\theta)\,\pi_k(R_{-\theta}D)$, it remains to show that $R^{-1}_\theta D$ is conjugated to a rotation.
Indeed,  if
$$
|\cos\theta| < \frac2{\eta_i+\eta_i^{-1}},
$$
in which our $\theta$ satisfies because of first inequality in \eqref{equ:theta}, then the characteristic polynomial of
\[R^{-1}_\theta D=
\left(\begin{matrix}
\eta_i\cos\theta&\eta_i^{-1}\sin\theta\\-\eta_i\sin\theta&\eta_i^{-1}\cos\theta
\end{matrix}\right),\]
has complex roots $(\cos\xi\pm\i\sin\xi)$, where $\xi\in[0,\frac\pi{2}]$ is chosen such that 
$$
\cos \xi= \frac{\eta_i+\eta_i^{-1}}2\cos\theta.
$$
A simple calculation shows that $R^{-1}_\theta D=P_iR_\xi P_i^{-1}$, where 
$$
P_i=\frac1{\sqrt{(\eta_i^{-1}\sin\theta\sin\xi)}}\begin{pmatrix}\eta^{-1}_i\sin\theta&0\\\cos\xi-\eta_i\cos\theta&\sin\xi\end{pmatrix}
$$
The matrix $P_i$ has determinant 1, so $\pi_k(P_i)$ is symplectic. Notice  that $(\eta_i^{-1}\sin\theta\sin\xi)^{1/2}$, $(\eta_i^{-1}\sin\theta)$, $\sin\xi$ and $(\cos\xi-\eta_i\cos\theta)$ go to zero with equal rates when 
$$
0<\theta<\xi\to 0.
$$ 
That is, $P_i$ is close to $\left(\begin{smallmatrix}1&0\\1&1\end{smallmatrix}\right)$ so $\|P_i^\pm\|$ is bounded from above by some constant $c$.
The conjugating matrices $P_m$ are of the form $P_iP$ where $P$ is the close to identify symplectic matrix given by Lemma \eqref{lemma:diagsymp}, so  $P_m\in\Sp(2d,\Rr)$ and  $\|P_m^{\pm}\|\leq c$.

Notice that $(\eta_i+\eta_i^{-1})/2\geq1$ for every $\eta_i>0$, so $1-\cos\xi<1-\cos\theta<|\eta_i-1|$.
Finally, there is constant $c'$ such that for any $k=1,\dots,d$,
$$
\|\pi_k(R_\xi)-I\| = 1-\cos\xi + \sqrt{1+\cos\xi} \sqrt{1-\cos\xi}
\leq c'|\lambda-1|^{1/2}.
$$
 
\end{proof}


\section{Proof of Theorem~\ref{thm:Frankspo}}
\label{sec:proof1}

We want to construct a perturbation $g$ of $f$ around $p\in M$ which realizes a matrix $\hat{A}_\pi D_pf$ close to $D_pf$.
We choose 
$$
g=\varphi^{-1}\circ h\circ\varphi\circ f,
$$
where $(U,\varphi)$ is a local chart (from the splitting theorem) at $f(p)$ with $\varphi(f(p))=0$, 
and $h$ is a Poisson diffeomorphism of $\Rr^{2d+n}$ that fixes the origin. 
Therefore, $D_0h=C\,D_pg\,D_pf^{-1} \, C^{-1}$ with $C=D_{f(p)}\varphi$ Poisson.
Let 
$$
A_\pi
=C\,(\hat{A}_\pi D_pf) \,D_pf^{-1}\,C^{-1} 
=I+C\, \hat{A}_\pi \, C^{-1}.
$$
By Lemma~\ref{lemma:decomposition symp} we write $A=A_1\dots A_{4d}$ with 
$A_k=P_k R_{\alpha_k} P_k^{-1}$.
We want to use Lemma~\ref{lemma K} for each $k$ by constructing $K_k$ as in \eqref{Ham K} for $\alpha_k$ corresponding to a rotation in the coordinates $i(k)$ and $d+i(k)$.
This guarantees the existence of Poisson map $h_k=(P_k)_\pi\circ\varphi_{K_k}^1\circ (P_k)_\pi^{-1}$ fixing the origin and satisfying $D_0h_k=(A_k)_\pi$.
So, the choice $h=h_1\circ \dots\circ h_{4d}$ implies that $D_0h=A_\pi$.
Furthermore, 
\begin{equation*}
\begin{split}
h-\id &= \sum_{k=1}^{4d}(h_k-\id)\circ h_{k+1}\circ\dots\circ h_{4d}, \\
Dh-I &= \sum_{k=1}^{4d}(Dh_k -I)\circ h_{k+1}\circ\dots\circ h_{4d}.
\end{split}
\end{equation*}
Notice that, by the integral formula $\varphi_{K}^t-\id=\int_0^tX_{K}\circ\varphi_{K}^s\,ds$,
$$
\|\varphi_K^t-\id\|_{C^0}\leq \|K\|_{C^1}.
$$
Take $T$ such that $\|D\varphi_K^T-I\|_{C^0}=\max_{s}\|D\varphi_K^s-I\|_{C^0}$.
Then, from $D\varphi_{K}^t-I=\int_0^tDX_{K}\circ\varphi_{K}^s(D\varphi_{K}^s-I)\,ds+\int_0^tDX_{K}\circ\varphi_{K}^s\,ds$, one gets
$$
\|D\varphi_K^T-I\|_{C^0}\leq \|DX_{K}\|_{C^0} \|D\varphi_K^T-I\|_{C^0} + \|DX_{K}\|_{C^0} 
$$
and
$$
\|D\varphi_K^1-I\|_{C^0}\leq \frac{\|K\|_{C^2}}{1-\|K\|_{C^2}}.
$$
We obtain the following estimates from Lemmas~\ref{lemma K} and~\ref{lemma:decomposition symp},
$$
\|h-\id\|_{C^1} \leq \sum_{k=1}^{4d}\|h_k-\id\|_{C^1} 
\leq \sum_{k=1}^{4d} \frac{c_1|\alpha_k|}{1-c_2|\alpha_k|}
\leq
c_3 \delta^{1/2}
$$
as long as $\delta$ is small enough.
Notice that $|\alpha_k| \leq c_4 \|R_{\alpha_k}-I\|$ for some constant $c_4>0$ whenever $|\alpha_k|$ is close to zero.

Finally,
\begin{equation*}
\begin{split}
\|g-f\|_{C^0} &= \| (\varphi^{-1}\circ h-\varphi^{-1})\circ\varphi \circ f\|_{C_0} \\
&\leq \|D\varphi^{-1}\|_{C^0} \| h-\id\|_{C^0}
\end{split}
\end{equation*}
and
\begin{equation*}
\begin{split}
\|Dg-Df\|_{C^0} &= \| (Dg\,(Df)^{-1}-I)Df\|_{C_0} \\
&\leq \|D\varphi^{-1}\|_{C^0} \|D\varphi\|_{C^0} \|Df\|_{C^0}   \| Dh-I\|_{C^0}.
\end{split}
\end{equation*}
Hence,
\begin{equation*}
\begin{split}
\|g-f\|_{C^1} &=\max\left\{ \|g-f\|_{C^0},\|Dg-Df\|_{C^0} \right\} \\
&\leq \max\left\{ 1, \|D\varphi\|_{C^0} \|Df\|_{C^0} \right\}   \|D\varphi^{-1}\|_{C^0} \| h-\id\|_{C^1}.
\end{split}
\end{equation*}
This is less than $\varepsilon$ as long as $\delta$ is small enough.


\section{Hamiltonian flows in Poisson manifolds}\label{sec:Hamiltonian}

Consider $\Rr^{2d+n}$ equipped with the Poisson structure $\pi_0$.
The Hamiltonian 
$$
H_0(x,y,z)=y_{1},
\quad
x\in\Rr^d,y\in\Rr^d,z\in\Rr^n,
$$
corresponds to the constant vector field $X_{H_0}=(1,0,\dots,0)$ and the flow 
$$
\varphi_{H_0}^t=\id + t X_{H_0}.
$$
Therefore, it has the orbit segment $\Gamma_0=[0,1]\times\{0\}$ corresponding to the orbit of the origin.

Fix the transversals $\Sigma_0=\{x_1=0,y_1=0,z=0\}$ and $\Sigma'_0=\{x_1=1,y_1=0,z=0\}$
at the edges of $\Gamma_0$.
The Poincaré map $P_{H_0}\colon\Sigma_0\to\Sigma'_0$ is $P_{H_0}=\id+(1,0,\dots,0)$.

Given $r>0$ define
$$
V_{r}=\{(x,y,z)\in\Rr^{2d+n}\colon 0\leq x_1\leq 1, |y_{1}|< r, \|(\hat x,\hat y,z)\| < r\}.
$$
Here and in the following we use the notations $\hat x=(x_2,\dots,x_{d+1})$ and $\hat y=(y_{2},\dots,y_{d+1})$.

\begin{proposition}\label{local Franks Ham}
For any $\varepsilon>0$ there is $\delta>0$ such that
$$
\{A\in \Sp(2d,\Rr)\colon \|A-I\|<\delta\}
\subset
\{D_pP_{H}\in\Sp(2d,\Rr)\colon H\in B_\varepsilon(H_0,D)\},
$$
where $D=(\Rr^{2d+2+n}\setminus B_\varrho)\cup \Gamma_0$ for any open ball $B_\varrho$ centered at the origin with radius $\varrho>0$.
\end{proposition}

\subsection{Proof of Proposition~\ref{local Franks Ham}}

Given $\alpha\in\Rr$ and $1\leq i\leq d$.
Consider the $C^\infty$ function $\widetilde H\colon\Rr^{2d+2+n}\to\Rr$,
\begin{equation}\label{tilde H}
\widetilde H(x,y,z) = 
\ell(2x_1-1) 
\ell(y_{1})
K_i(\hat x,\hat y,z),
\end{equation}
where $\ell$ is the bump function defined at~\eqref{def:ell}.
We write the function $K_{i}$ as $K_i$ in~\eqref{Ham K} rotating the coordinates $i$ and $d+i$.

\begin{lemma}\label{lemma:PH=R}
Let $H=H_0+\widetilde H$.
Then, 
\begin{enumerate}
\item $X_{H}=X_{H_0}$ on $(\Rr^{2d+n} \setminus V_1)\cup\Gamma_0$,
\item if $0<r\leq1/4$, $P_H\colon V_r\cap\Sigma_0\to\Sigma'_0$ is given by $P_H=\pi_i(R_\alpha)$,
\item $\|H-H_0\|_{C^2} \leq c|\alpha| $ for some constant $c>0$.
\end{enumerate}
\end{lemma}

\begin{proof}
For $(x,y)\in V_{r}$ with $r\leq1/4$, the Hamiltonian equations of motion $(\dot x,\dot y,\dot z)=X_{H}(x,y,z) = \hat\Jj \nabla H(x,y,z)$ are
\begin{equation*}
\begin{split}
\dot x_1 &=1 \\
\dot y_{1} &=- \alpha \ell'(2x_1-1) (x_{i}^2+y_{i}^2) \\
\begin{pmatrix}\dot x_{i} \\ \dot y_{i} \end{pmatrix} &=
\alpha \, \ell(2x_1-1)
\begin{pmatrix} y_i \\ -x_{i} \end{pmatrix} 
\end{split}
\end{equation*}
and $\dot x_j=\dot y_j=\dot z_k=0$ for $j\neq i$ and $k=1,..,n$.
It is easy to check that $\frac d{dt}(x_i^2+y_{i}^2)=0$.
So, on this domain, the Hamiltonian flow $(x(t),y(t))=\varphi_{H_i}^t(x,y)$ is 
\begin{equation*}
\begin{split}
x_1(t) &= x_1+t \\
y_{1}(t) &=  y_{1}-\alpha\,\left[\ell(2x_1+2t-1)-\ell(2x_1-1)\right] \, \frac{(x_{i}^2+y_{i}^2)}{2}\\
\begin{pmatrix} x_{i}(t) \\ y_{i}(t) \end{pmatrix} &=
R_{\theta(t,x_1)}
\begin{pmatrix} x_i \\ y_{i} \end{pmatrix},
\end{split}
\end{equation*}
where $\theta(t,x_1)=\alpha \int_0^t\ell(2x_1+2s-1)ds$, while in the remaining coordinates the flow is constant.

For $(x,y,z)\in V_r\cap \Sigma_0$ we have $y_{1}=0$, so $|y_{1}(t)|\leq r^2 |\alpha| < r$ if $r<|\alpha|^{-1}$ and  $\varphi_{H}^t(V_{r}\cap \Sigma_0)\subset V_{r}$, and for $(0,\hat x,0,\hat y,0)\in V_r\cap \Sigma_0$ we have $\varphi_H^1(0,\hat x,0,\hat y,0)=(1,\hat x_\alpha,0,\hat y_\alpha,0),$ where $(\hat x_\alpha,\hat y_\alpha)=\pi_i(R_\alpha)(\hat x,\hat y)$.
In particular, $\varphi_{H}^t(0)=(t,0)$ implies that $\Gamma_0$ is an orbit segment of $H$ with transversals $\Sigma_0$ and $\Sigma'_0$ at the edges.
Therefore, whenever $(x,y,z)\in V_{r}\cap \Sigma_0$ we have $\varphi^1_{H}(x,y,z)\in\Sigma'_0$.  So, $P_H\colon  V_{r}\cap \Sigma_0\to\Sigma'_0$ just acts on the coordinates $i$ and $d+i$ by rotating an angle $\theta(1,0)=\alpha$ i.e. $P_H(\hat x,\hat y)=\pi_i(R_\alpha)(\hat x, \hat y)$.

Finally, we need to estimate the $C^2$-norm of the perturbation.
It is simple to check that
\begin{equation}\label{estimate:1}
\|H-H_0\|_{C^1} = \max\{ \|H-H_0\|_{C^0} , \| X_H-X_{H_0}\|_{C^0} \} \leq r|\alpha| \,\|\ell\|_{C^1}.
\end{equation}
Writing $\rho=\|(\hat x,\hat y,z)\|^2/2$ and $\rho_i=(x_i^2+y_i^2)/2$, the second order derivatives are the following:
\begin{equation*}
\begin{split}
\frac{\partial^2H}{\partial x_1^2} &=4\alpha \ell''(2x_1-1)\ell(y_1) \ell(\rho) \rho_i \\
\frac{\partial^2H}{\partial x_1\partial y_1} &=2\alpha \ell'(2x_1-1)\ell'(y_1) \ell(\rho) \rho_i \\
\frac{\partial^2H}{\partial x_1\partial w} &=2\alpha \ell'(2x_1-1)\ell(y_1) [\ell'(\rho) \rho_i+\ell(\rho)] w\\
\frac{\partial^2H}{\partial x_1\partial u} &=2\alpha \ell'(2x_1-1)\ell(y_1)\ell'(\rho)\rho_i u\\
\frac{\partial^2H}{\partial y_1^2} &=\alpha \ell(2x_1-1)\ell''(y_1)\ell(\rho)\rho_i \\
\frac{\partial^2H}{\partial y_1\partial w} &=\alpha \ell(2x_1-1)\ell'(y_1)[\ell'(\rho)\rho_i+\ell(\rho)]w \\
\frac{\partial^2H}{\partial y_1\partial u} &=\alpha \ell(2x_1-1)\ell'(y_1)\ell'(\rho)\rho_i u
\end{split}
\end{equation*}
and
\begin{equation*}
\begin{split}
\frac{\partial^2H}{\partial x_i\partial y_i}& =\alpha \ell(2x_1-1)\ell(y_1) [\ell''(\rho)\rho_i+2\ell'(\rho)] x_iy_i \\
\frac{\partial^2H}{\partial w^2} &=\alpha \ell(2x_1-1)\ell(y_1) [\ell'(\rho)\rho_i+\ell(\rho) +\ell''(\rho)\rho_i w^2+2\ell'(\rho)w^2] \\
\frac{\partial^2H}{\partial w\partial u}& =  \alpha \ell(2x_1-1)\ell(y_1) [\ell'(\rho)\rho_i+\ell(\rho)]wu \\
\frac{\partial^2H}{\partial u^2}& = \alpha \ell(2x_1-1)\ell(y_1) [\ell(\rho) +\ell'(\rho)u^2]\rho_i \\
\frac{\partial^2H}{\partial u\partial v}& = \alpha \ell(2x_1-1)\ell(y_1) \ell'(\rho)\rho_i uv, 
\end{split}
\end{equation*}

where $w$ stands for $x_i$ or $y_i$, $u$ and $v$ replace $x_j$, $y_j$ or $z_k$ with $j\not=i$ and $u\not=v$. So, for some constant $c>0$,
\begin{equation}\label{estimate:2}
\begin{split}
\|D^2(H-H_0)\|_{C^0} & \leq c |\alpha| \, \|\ell\|_{C^2},
\end{split}
\end{equation}
in which together with \eqref{estimate:1} yields part (3) of the lemma. 

\end{proof}

Consider a finite set of matrices $A_k=P_kR_kP_k^{-1}\in\Sp(2d,\Rr)$, $k=1,\dots,N$, with $P_k\in\Sp(2d,\Rr)$, $R_k=\pi_{i(k)}(R_{\alpha_k}),$  $\alpha_k\in\Rr$ and $1\leq i(k)\leq d$.
We write $\widetilde H_{i(k)}$ for the Hamiltonian in~\eqref{tilde H} for a function $K_{i(k)}$.

\begin{lemma}\label{lemma:combination}
There is $0<r<1$ such that
$$
H=H_0 + N\sum_{k=1}^{N} \widetilde H_{i(k)}\circ \Phi(P_{k}) \circ T_{k},
$$
where $T_{k}(x,y)=(N x_1-k+1,x_{2},\dots,x_d,y_1,\dots,y_{d})$ and $\Phi$ is defined at \eqref{def:Phi},
verifies 
\begin{enumerate}
	\item $X_{H}=X_{H_0}$ on $(\Rr^{2d+n}\setminus V_1)\cup \Gamma_0$,
	\item $P_H\colon V_r\cap\Sigma_0\to \Sigma'_0$, $P_H=A_{N}\cdots A_1$,
	\item $\|H-H_0\|_ {C^2} \leq c\,N^2\, \max_{k=1,\dots,N}\{(\max\{1,\|P_k\|^2\})|\alpha_k|\}$, for some constant $c>0$.
\end{enumerate}
\end{lemma}

\begin{proof}
For each $k=1,\dots, N+1$ let 
$$
\Sigma_k=\left\{(x,y,0)\in\Rr^{2d+n}\colon x_1=\frac {k-1}{N}, y_{1}=0\right\}.
$$
If $k-1 \leq Nx_1\leq k$ with $1\leq k\leq N$, then $H=H_0+N H_{i(k)}\circ \Phi(P_k)\circ T_k$.
So, using Lemma~\ref{lemma:PH=R} and taking $r$ sufficiently small, the Poincaré map of $H$ between the transversals to $\Gamma_0$ given by $P_H\colon\Sigma_{k}\to\Sigma_{k+1}$ is $P_H=A_k$.
Recall that for any $f\in C^2(\Rr^{2d+n})$ we have that $\varphi_{f\circ P}^t=P^{-1}\circ\varphi_f^t\circ P$ where $P$ is a linear map.

The Poincar\'e map for the transversals at the edges of $\Gamma_0$, $P_H\colon\Sigma_0\to\Sigma_{N+1}=\Sigma_0'$, is the composition of the above maps.
That is, $P_H=A_{N}\cdots A_1$ on $V_r\cap\Sigma_0$.

Finally, the norm can be estimated also by using Lemma~\ref{lemma:PH=R}, 
$$\|H-H_0\|_ {C^2} \leq c \, \|\ell\|_{C^2} N^2 \max_{k=1,\dots,N}\{(\max\{1,\|P_k\|^2\})|\alpha_k|\},$$ for some constant $c>0$.
\end{proof}

Considering Lemmas~\ref{lemma:decomposition symp},~\ref{lemma:PH=R} and~\ref{lemma:combination}, in order to complete the proof of Proposition~\ref{local Franks Ham} it remains to show that $H$ can be taken equal to $H_0$ outside any ball of radius $\varrho>0$.
Define the $\varrho$-open ball $B_\varrho\subset\Rr^{2d+2+n}$ around the origin.
Consider the Hamiltonian $\widetilde H=\varrho H\circ \psi$ with $\psi(x,y)=(\varrho^{-1} x,\varrho^{-1} y)$.
Then $X_{\widetilde H}=X_H\circ\psi$ implies that $\varphi_{\widetilde H}^t=\varphi_H^{t/\varrho}$.
So, up to a time change, the dynamics are the same.


\subsection{Poisson Flowbox coordinates}

For a manifold $M$ and a submanifold $N$, we will denote the annihilator of $TN$ inside $T^\ast M$  by 
$$
TN^\circ=\{\xi\in T^\ast M \colon \xi(v)=0, v\in  TN\}.
$$

Let $(M,\pi)$ be a Poisson manifold. For a given $H\in C^s(M)$ we define the corresponding Hamiltonian vector field as 
$$
X_H=\{\cdot,H\}=\pi(\cdot,d H).
$$
In particular, if $M=\Rr^{2d+n}$ with coordinates $(y_1,...,y_{2d+n})$ and the standard Poisson structure 
$$
\pi_0=\sum_{i=1}^d \frac{\partial}{\partial y_i}\wedge\frac{\partial}{\partial y_{d+i}},
$$
then $H_0(y)=y_{d+1}$ yields $X_{H_0}=\frac{\partial}{\partial y_1}$.

The following result is the version of a straightening theorem in the Poisson context, (cf.~\cite{AM,BLD2008} for the symplectic case).

\begin{theorem}[Poisson flowbox coordinates]\label{thm:flowbox}
Let $(M^{2d+n},\pi)$ be a $C^s$-Poisson manifold, a Hamiltonian $H\in C^s(M,\Rr)$, $s\ge 2$ or $s=\infty$, and $x\in M$ such that the rank of $\pi$ is constant in a neighborhood of $x$. 
If $X_H(x)\neq 0$, there exist a neighborhood $U\subset M$ of $x$ and a local $C^{s-1}$-Poisson diffeomorphism $g\colon (U,\pi)\to(\Rr^{2d+n},\pi_0)$ such that $H=H_0\circ g$ on $U$.
\end{theorem}

\begin{proof}
Fix $e=H(x)$. 
Since $X_H(x)\neq 0$ one can find a coordinate patch $(U,(q_1,...,q_{2d+n}))$ centered at $x$, such that $X_H=\frac{\partial}{\partial q_1}$. In the neighborhood $U$ we have:
\[ \{H,q_1\}=\pi(d H,d q_1)=X_H(q_1)=\frac{\partial q_1}{\partial q_1}=1\]
We will denote $q_1$ by $G$ and the neighborhood $U$ will be allowed to remain as small as needed. For small enough $U$ one can define the transversal $\Sigma$ at point $x$ by 
\[\Sigma=G^{-1}(0)\cap U\]
which is a $C^s$ regular connected submanifold of dimension $2d+n-1$. Notice that $\{H,G\}=1$  holds in $U$. 

Locally there is a $C^s$ regular $(2d+n-2)$-dimensional hypersurface of $H^{-1}(e)$ where $H$ and $G$ are both constant: $\Sigma_e=\Sigma\cap H^{-1}(e)$. Notice that for $m\in\Sigma_e$,
\[T_m\Sigma_e=\{v\in T_M\colon  d H(v)(m)=d G(v)(m)=0\},\]
since $d H(X_G)(m)=-d G(X_H)(m)=1$, we have $X_G(m),X_H(m)\notin T_m\Sigma_e$ and 
\[T_mM=T_m\Sigma_e\oplus\real X_H\oplus \Rr X_G.\]
Also,
\[T^\ast_mM=T^\ast_m\Sigma_e\oplus\real d H(m)\oplus\real d G(m).\]

Consider the pointwise linear map $\pi^\sharp \colon T^\ast M\to TM$ given by 
$$
\xi(m)\mapsto\pi(m)(\xi(m),\cdot).
$$
Since the rank of $\pi$ is constant in the neighborhood around $x$ where $\Sigma_e$ is defined, 
showing that
\begin{equation}\label{pdcondition}
\pi^\sharp(TN^\circ)\cap TN=\{0\}
\end{equation}
implies that $\Sigma_e$ is a Poisson-Dirac submanifold of $M$ (cf.~\cite[\S 8]{CF04}). 
Poisson-Dirac submanifolds have a canonically induced  Poisson structure. 
In the language of Poisson brackets one can compute the induced Poisson structure $\pi_e$ on the submanifold $\Sigma_e$ as follows:
\[\pi_e( d F_1, d F_2)=\{F_1,F_2\}_e:=\{\hat{F_1},\hat{F_2}\}|_{\scriptstyle{\Sigma_e}},\]
where $\hat{F_i}\in C^\infty(M)$ are extensions of $F_i\in C^\infty(\Sigma_e)$ such that $ d F_i|_{\scriptstyle{\pi^\sharp(TN^\circ)}}=0$, $i=1,2$. 

Lets check that~\eqref{pdcondition} holds for $\Sigma_e$. Notice that by elementary linear algebra $T\Sigma_e^\circ$ is two dimensional. Furthermore, $H$ and $G$ are constant on $\Sigma_e$, i.e. $ d H,  d G\in T\Sigma_e^\circ$. 
Moreover, $ d H$ and $ d G$ are independent in $U$ so $T\Sigma_e^\circ=\real d H\oplus\real d G$. 
The definition of a Hamiltonian vector field gives us
\[\pi^\sharp( d H)=X_H
\quad
\text{and}
\quad
\pi^\sharp( d G)=X_G.\] 
It is now clear that~\eqref{pdcondition} holds. 

The corollary of Weinstein's splitting theorem for constant rank Poisson structures~\cite[Theorem 1.26]{poissonbook} assures us the existence of a local diffeomorphism $h\colon\Sigma_e\to \real^{2d+n-2}$ such that 
\begin{equation}\label{darboux}
h_\ast\pi_e=\pi_0^\prime
\quad
\text{where}
\quad
\pi_0^\prime=\sum_{i=2}^d\frac{\partial}{\partial y_i}\wedge\frac{\partial}{\partial y_{d+i}}.
\end{equation}
The next step is to extend the above Poisson coordinates from $\Sigma_e$ to $U$. 
For this purpose we use the parametrization by the flows $\phi^t_H$ and $\phi^t_G$ generated by
$X_H$ and $X_G$, respectively. 

Consider the function $G\circ\phi_H\colon U\times\real\to\real$, $(m,t)\mapsto\phi_H^t(m)$.
As $G\circ\phi_H^0(x)=0$ and 
\[\frac{ d}{ d t}G\circ\phi^t_H(x)|_{t=0}= d G(X_H)(x)=\{G,H\}=-1\neq0,\]
by the implicit function theorem we know that for $U$ small enough, there exist a unique $\tau\in C^{s-1}(U,\real)$ such that $G\circ\phi_H^{\tau(m)}=0$, i.e. $\phi_H^{\tau(m)}\in\Sigma$ for each $m\in U$. 
Moreover, $\phi_G^t$ preserves the level set of $G$ and 
\[\frac{ d}{ d t}H\circ\phi_G^t=\{H,G\}=1.\]
Thus, $H\circ\phi^t_G(m)=H(m)+t$ and in particular $H\circ\phi_G^{e-H(m)} (m)=e$. Hence, $\phi_G^{e-H(m)} (m)\in H^{-1}(e)$ for each $m\in U$.

So, we define $g\colon U\to \real^{2d+n}$ given by
\begin{equation}
\begin{split}
g(m) =& \left(-\tau(m),h_1\circ\phi_G^{e-H(m)}\circ\phi_H^{\tau(m)}(m),H(m),  \right. \\
& \left. h_2\circ\phi_G^{e-H(m)}\circ\phi_H^{\tau(m)}(m),h_3\circ\phi_G^{e-H(m)}\circ\phi_H^{\tau(m)}(m)\right),
\end{split}
\end{equation}
where $h=(h_1,h_2,h_3)$ as in~\eqref{darboux}, $h_i\colon\Sigma_e\to \real^{d-1}$, $i=1,2$, and $h_3\colon\Sigma\to  \real^n$. 
Clearly, $H_0\circ g=H$. 
The proof will be complete as soon as we show that $g$ is a $C^{s-1}$ Poisson diffeomorphism.  

It follows that $g$ is $C^{s-1}$ with inverse $g^{-1}\colon g(U)\to   U$,
\begin{equation}\label{inverse}
g^{-1}(y)=\phi_H^{y_1}\circ\phi_G^{y_{d+1}-e}\circ h^{-1}(\hat{y}),
\end{equation}
where $\hat{y}=(y_2,...,y_d,y_{d+2},...,y_{2d+n})$.
In addition, for $y\in g(U)$, 
\begin{align}\label{pushforward1}
g_\ast^{-1}X_{H_0}(y)&=\dot{\phi}_H^{y_1}\circ\phi^{y_{d+1}-e}_G\circ h^{-1}(\hat{y})\notag\\
&=X_H\circ\phi^{y_1}_H\circ\phi_G^{y_{d+1}-e}\circ h^{-1}(\hat{y})\\
&=X_H\circ g^{-1}(y).\notag
\end{align}
Equivalently, $g_\ast X_H=H_{H_0}$. 
Furthermore, notice that the map $F\mapsto X_F$ from $C^s(M)$ to the set of $C^{s-1}$ vector fields $\mathfrak{X}^{s-1}(M)$ is a Poisson map, i.e. 
$$
\{F_1,F_2\}\mapsto X_{\{F_1,F_2\}}=[X_{F_1},X_{F_2}],
\quad
F_1,F_2\in C^S(M),
$$
cf.~\cite[{Proposition 1.4}]{poissonbook}. 
In $U$, we have $\{H,G\}=1$ thus $[X_H,H_G]=X_{\{H,G\}}=0$. 
This means that $\phi_H^{t_1}\circ\phi^{t_2}_G=\phi^{t_2}_G\circ\phi_H^{t_1}$.  
Using this fact and  
\begin{align}\label{pushforward2}
g_\ast^{-1}\frac{\partial}{\partial y_{d+1}}(y)&=\dot{\phi}^{y_{d+1}-e}_G\circ\phi_H^{y_1}\circ h^{-1}(\hat{y})\notag\\
&=X_G\circ\phi_G^{y_{d+1}-e}\circ\phi^{y_1}_H\circ h^{-1}(\hat{y})\\
&=X_G\circ g^{-1}(y),\notag
\end{align}
which is a similar calculation as in~\eqref{pushforward1}, we obtain $g_\ast(X_G)=\frac{\partial}{\partial y_{d+1}}$.

Notice that on $\Sigma_e$ we have
\begin{align*}
&(g^\ast  d y_1)(X_H)= d y_1(g_\ast  X_H)= d y_1(X_{H_0})=1\\
&(g^\ast  d y_1)(X_G)= d y_1(g_\ast  X_G)= d y_1(\frac{\partial}{\partial y_{d+1}})=0\\
&(g^\ast  d y_1)(v)= d y_1(g_\ast  v)= d y_1(h_\ast v)=0,
\quad v\in T_m\Sigma_e.
 \end{align*}
On the other hand,
\begin{align*}&- d G(m)(X_H)=-\{G,H\}(m)=1,\\
&- d G(m)(X_G)=-\{G,G\}(m)=0,\\
&- d G(m)(v)=0,\qquad v\in T_m\Sigma_e,
\end{align*}
by elementary linear algebra.
Then, $g^\ast  d y_1(m)=- d G(m)$ for each $m\in T_e\Sigma$. 
Similarly, $g^\ast dy_{d+1}(m)= d H(m)$. We also have 
$g^\ast d y_j(m)=h^\ast d y_j(m)\in T_m^\ast\Sigma_e$ for every $j\notin\{1,d+1\}$. 
Furthermore, taking in addition $k\notin\{1,d+1\}$,
\begin{align*}
&(g_\ast\pi)( d y_j, d y_k)=(h_\ast)\pi( d y_j, d y_k)=\pi_0( d y_j, d y_k),\\
&(g_\ast\pi)( d y_1, d y_{d+1})=\pi(g^\ast d y_1,g^\ast d y_{d+1})=\pi(- d G, d H)=-\{G,H\}=1,\\
&(g^\ast\pi)( d y_1, d y_j)=\pi(- d G,g^\ast( d y_j))=X_G(g^\ast( d y_j))=0,\\
&(g^\ast\pi)( d y_{d+1}, d y_j)=\pi( d H,g^\ast( d y_j))=-X_H(g^\ast( d y_j))=0.
\end{align*}
Therefore, $g_\ast(\pi)$  has to be the canonical Poisson structure $\pi_0$, i.e. $g_\ast\pi=\pi_0$ on $\Sigma_e$.

The inverse of $g$ in~\eqref{inverse} shows that every point in $U$ can be reached by the flows $\phi_G$ and $\phi_H$, consecutively. Fix a point $m\in\Sigma_e$ and $t_1,t_2\in\real$ such that $m^\prime=\phi_H^{t_2}\circ\phi_G^{t_1}(m)\in U$. We will restrict to a small neighborhood around $m$ in which $\phi^{t_2}_H\circ\phi^{t_1}_G$ does not take us outside $U$. Now, 
$$
G(\phi^{\tau(m)-t_2}_H(m^\prime))=G(\phi_G^{t_1}\circ\phi^{\tau(m)}_H(m))=G(\phi^\tau_H(m))=0,
$$
where we used the fact that $X_H$ and $X_G$ commute and $G$ is constant along orbits of $X_G$. 
Thus,
\begin{equation}\label{tau1}
\tau(m^\prime)=\tau(m)-t_2.
\end{equation}
Furthermore,
\begin{equation}\label{calcu1}
H(\phi_H^{t_2}\circ\phi_G^{t_1}(m))=H(\phi_G^{t_1}(m))=H(m)+t_1
\end{equation}
using the fact that $\{H,G\}=1$.
By the definition of $g$,~\eqref{tau1} and~\eqref{calcu1}, we get
\begin{equation}\label{calcu2}
g(m^\prime)=g(\phi_H^{t_2}\circ\phi_G^{t_1}(m))=g(m)+(t_2,0,\dots,0,t_1,0,\dots,0).
\end{equation}
 Notice that the flow of a Hamiltonian vector field, for any fixed time, is a Poisson map, so both $\phi_H^{t_2},\phi_G^{t_1}$ are Poisson maps in $(M,\pi)$.
In addition, translations in $(\real^{2d+n},\pi_0)$ are Poisson.  
Finally,~\eqref{calcu2} together with $g_\ast\pi=\pi_0$ on $\Sigma$ completes the proof.
\end{proof}

\subsection{Proof of Theorem~\ref{thm:Franks_symp}}

By Theorem~\ref{thm:flowbox} we can locally reduce the problem to the case of Proposition~\ref{local Franks Ham}.

\section*{Acknowledgements}

The authors were supported by Funda\c c\~ao para a Ci\^encia e a Tecnologia through the Program POCI 2010 and the Project ``Randomness in Deterministic Dynamical Systems and Applications'' (PTDC-MAT-105448-2008).


\end{document}